\documentclass{aims}
\usepackage{amsmath}
  \usepackage{paralist}
  \usepackage{graphics} 
  \usepackage{epsfig} 
\usepackage{graphicx}  \usepackage{epstopdf}
 \usepackage[colorlinks=true]{hyperref}
\hypersetup{urlcolor=blue, citecolor=red}

\def\currentvolume{X}
 \def\currentissue{X}
  \def\currentyear{200X}
   \def\currentmonth{XX}
    \def\ppages{X--XX}

\newtheorem{theorem}{Theorem}[section]

\newtheorem{lemma}[theorem]{Lemma}

\newtheorem{problem}{Problem}[section]
\theoremstyle{definition}

\newtheorem{remark}{Remark}

\newcommand{\eps}{\varepsilon}
\newcommand{\dps}{\displaystyle}

\newcommand{\fr}{\frac}
\newcommand{\pa}{\partial}

\title[Supersonic-sonic patch arising from the Frankl problem] 
      {On a supersonic-sonic patch arising from the Frankl problem in transonic flows}

\author[Yanbo Hu and Jiequan Li]{}


 \email{yanbo.hu@hotmail.com}
 \email{li\_jiequan@iapcm.ac.cn}

\thanks{The first author is supported by the Zhejiang Provincial Natural Science Foundation (No. LY21A010017). The second author is supported by the Natural Science Foundation of China (Nos: 11771054, 91852207,12072042),  National Key Project(GJXM92579) and Foundation of LCP}

\thanks{$^*$ Corresponding author: Jiequan Li}

\begin{document}
\maketitle

\centerline{ Dedicated to  the celebration of the 80th birthday of Professor Shuxing Chen}

\

\centerline{\scshape Yanbo Hu}
\medskip
{\footnotesize
 \centerline{Department of Mathematics, Hangzhou Normal University}
   \centerline{ Hangzhou, 311121, PR China}
} 

\medskip

\centerline{\scshape Jiequan Li$^*$}
\medskip
{\footnotesize
 \centerline{ Laboratory of Computational Physics, Institute of Applied Physics}
   \centerline{and Computational Mathematics, Beijing, 100088, China; }
   \centerline{Center for Applied Physics and Technology, Peking University, 100871, China}
}

\bigskip

 \centerline{(Communicated by the associate editor name)}

\begin{abstract}
We construct a supersonic-sonic smooth patch solution for the two dimensional steady Euler equations in gas dynamics. This patch is extracted from the Frankl problem in the study of transonic flow with local supersonic bubble over an airfoil. Based on the methodology of characteristic decompositions, we establish the global existence and regularity of solutions in a partial hodograph coordinate system in terms of  angle variables.
The original problem is solved by transforming the solution in the partial hodograph plane back to that in the physical plane. Moreover, the uniform regularity of the solution and the regularity of an associated sonic curve are also verified.
\end{abstract}

\section{Introduction}\label{s1}

Supersonic bubbles are ubiquitous in transonic flow problems, which are a kind of the important phenomena in compressible fluid dynamics. In the famous book (Supersonic Flow and Shock Waves, 1948, Page 370, \cite{Courant}), Courant and Friedrichs described a supersonic bubble arising in a duct: Suppose the duct walls are flat except for a small inward bulge at some section. If the entrance Mach number is not much below the value one, the flow becomes supersonic in a finite region adjacent to the bulge and is again purely subsonic throughout the exit section.
Such supersonic bubbles also arise over the airfoil similarly to those near the nozzle throat shown in Fig. 1. For more examples of transonic flows, we refer the reader to the monograph of Kuz'min \cite{Kuzmin}.

\begin{figure}[htbp]
\begin{center}
\includegraphics[scale=0.5]{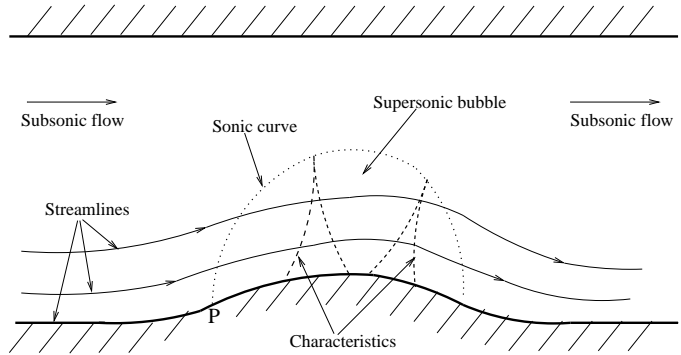}
\caption{\footnotesize Transonic phenomena in a duct.}
\end{center}
\end{figure}

The transonic flow problem described in Fig. 1 has been studied by many researchers, but the existence of its solutions is still open in the ``global'' transonic sense mathematically until now. See the review paper by Chen \cite{ChenS}. Some exact transonic solutions to the asymptotic equation of the potential equation, e.g., the Ringleb flow, were presented in \cite{Cole}. In a well-known paper \cite{Morawetz56}, Morawetz showed that the smooth flows for this transonic problems are unstable in general. Bers \cite{Bers1} studied the existence conditions of continuous solutions across the sonic line. The existence results of weak solutions were given in \cite{Morawetz85, Chen1} by using the method of compensated-compactness. For the fundamental solutions of the linear Tricomi and Keldysh operators, one may consult \cite{BG, ChenS2},  and also look up related results in \cite{ChenS1, ChenS3, ChenS5} etc.
In the past few years, many efforts have been made to investigate the existence of global subsonic-sonic solutions for the steady Euler equations. Xie and Xin \cite{Xie-Xin1} established
the existence of solutions in a subsonic-sonic part of the 2-D nozzle.
The global existence of subsonic-sonic flows for the full Euler equations was obtained by Chen {\em et al}  \cite{Chen4} in the compensated-compactness framework. Wang and Xin recently constructed a smooth transonic solution of Meyer type in Laval nozzles in \cite{Wang19}. On the other side, a local classical supersonic solution was established near a given sonic curve for the steady isentropic irrotational Euler equations by Zhang and Zheng \cite{ZhangT1}. In \cite{Hu-Li1}, Hu and Li verified the existence of sonic-supersonic classical solutions for the steady full Euler equations. Moreover, the existence of semi-hyperbolic patches of solutions to the isentropic irrotational Euler equations were provided in \cite{Lim, SWZ}.

For the transonic flow problem described as in Fig. 1, there may exist a transonic shock in the downstream flow by the classic result of Morawetz \cite{Morawetz56}. However, it is of importance in practice to construct shock-free transonic flows. In \cite{Frankl}, Frankl explored the transonic flow with supersonic bubbles over a symmetric airfoil and suggested that a smooth transonic flow may exist if a part of the airfoil is free of boundary conditions. As illustrated in Fig. 2, with the slip condition on the arcs $\widehat{PE}$ and $\widehat{FQ}$, the Frankl problem is formulated to find airfoil's arc $\widehat{EF}$ for the correctness of the problem in the class of smooth solutions. In \cite{Morawetz54}, Morawetz discussed the uniqueness of the Frankl problem for a second-order linear equation which is derived from the steady isentropic irrotational Euler system by the hodograph transformation.
A similar uniqueness result in the physical plane was given by Cook \cite{Cook}. Kuz'min proposed a modified Frankl problem in which a velocity distribution instead of the slip condition is prescribed on the arcs $\widehat{PE}$ and $\widehat{FQ}$ in \cite{Kuzmin}. From the physical viewpoint, this kind of problems describe the transonic flows past permeable (porous or perforated) boundaries. The solvability of a nonlinear perturbation problem, originated from the modified Frankl problem, was studied by Kuz'min \cite{Kuzmin0}, who also discussed the uniqueness of the modified Frankl problem to a linearized equation of the von Karman equation in a finite domain  \cite{Kuzmin1}.

\begin{figure}[htbp]
\begin{center}
\includegraphics[scale=0.55]{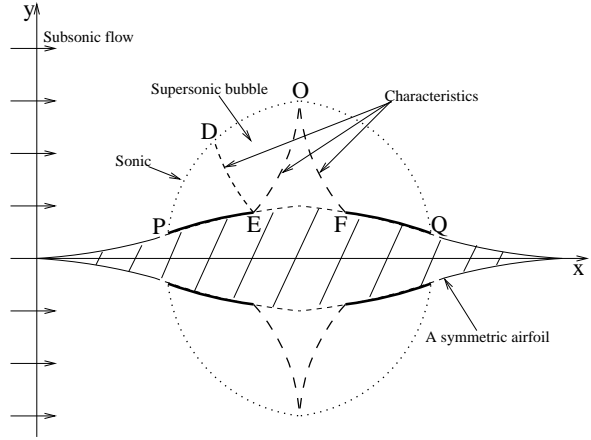}
\caption{\footnotesize The modified Frankl problem. With a velocity distribution on the arcs $\widehat{PE}$ and $\widehat{FQ}$, find  an airfoil's arc $\widehat{EF}$, free of boundary conditions,  for the correctness of the problem in the class of smooth solutions.}
\end{center}
\end{figure}

In this paper, we are interested in the existence of smooth solutions of the modified Frankl problem for the steady isentropic compressible Euler equations in two dimensions (2-D). Given a velocity distribution on the arc $\widehat{PE}$ such that $P$ is sonic, our problem is to find a sonic curve starting from $P$ and build a smooth supersonic solution in the region bounded by this sonic curve and $\widehat{PE}$. We hope to extend this solution up to the positive characteristic curve $\widehat{EO}$ by solving a free boundary value problem and then construct a global smooth supersonic flow by the symmetry of the airfoil in future work. This kind of patches, in fact, has more applications, e.g., the problem of transonic flow around a convex corner \cite{Kuzmin, Vaglio}.

We adopt the 2-D steady isentropic Euler equations,
\begin{align}\label{1.1}
\left\{
\begin{array}{l}
  (\rho u)_x+(\rho v)_y=0, \\
  (\rho u^2+p)_x+(\rho uv)_y=0, \\
   (\rho uv)_x+(\rho v^2+p)_y=0
\end{array}
\right.
\end{align}
where $\rho$ is the density, $(u,v)$ are the velocity, $p$ is the pressure satisfying the polytropic gas equation of state $p=A\rho^\gamma$ for some positive constant $A$, and $\gamma>1$ is the adiabatic index. For irrotational flows, that is, $u_y=v_x$, system \eqref{1.1} can be reduced to
\begin{align}\label{1.2}
\left\{
\begin{array}{l}
   (c^2-u^2)u_x-uv(u_y+v_x)+(c^2-v^2)v_y=0, \\
   u_y-v_x=0,
\end{array}
\right.
\end{align}
coupled with the Bernoulli law
\begin{align}\label{1.3}
  q^2+\fr{2c^2}{\gamma-1}=B_0.
\end{align}
Here $c=\sqrt{p'(\rho)}$ denotes the sound speed, $q=\sqrt{u^2+v^2}$ denotes the flow spend and $B_0$ is a positive integration constant depending on the flow. The eigenvalues of \eqref{1.2} are
\begin{align}\label{1.4}
 \Lambda_{\pm}=\frac{uv\pm c\sqrt{q^2-c^2}}{u^2-c^2},
\end{align}
from which we see that system \eqref{1.2} is of mixed-type: supersonic for $M>1$, subsonic for $M<1$ and sonic for $M=1$. Here $M$ is the Mach number defined by $M=q/c$. The set of points at which $M=1$ is called the sonic curve.

Since $P$ is a sonic point in the problem, that is $M=1$ at $P$,  we need solve a degenerate hyperbolic boundary value problem. A main difficulty of the paper is the singularity caused by the sonic degeneracy. In a recent paper \cite{Hu-Li2}, we investigated the existence of smooth solution of a supersonic-sonic patch bounded by a streamline and a characteristic curve. This result needs to be assigned data not only on the streamline but also on the characteristic curve such that level curves of $(M-1)$ can be taken as the ``Cauchy supports'' to obtain a global solution up to a sonic curve. Generally speaking, the data on the characteristic curves in the supersonic region are unknown {\em a priori}, which restricts  the application of the results in \cite{Hu-Li2}. Similar situations occur as the semi-hyperbolic patch problems considered in \cite{Lim, SWZ}. In contrast to the previous problems, the corresponding ``Cauchy supports'' can not be directly taken for the current problem  because only the data on the streamline are given. To overcome this difficulty, we analyze this problem not in the physical plane as handled in the previous works but in a partial hodograph coordinate system in terms of  angle variables. By introducing a set of dependent variables, the Euler equations \eqref{1.2} can be transformed as a linear system in the partial hodograph plane. We then establish the global existence of smooth solutions for the degenerate linear problem and further transform this solution back to the physical plane so as to solve the original problem. We comment that, as well-known, system \eqref{1.2} can be linearized by using the hodograph method (i.e. switching the roles of $(u,v)$ with $(x, y)$), but it is difficult in taking on boundary conditions and in returning back to the original variables due to the sonic degeneracy. The idea of using angle variables as an auxiliary coordinate system was adopted by Hu and Li \cite{Hu-Li1} to deal with the steady full Euler equations. It turns out that this coordinate system works effectively for the current and related problems.
\vspace{0.2cm}

The main result of the paper is stated as follows.

\begin{theorem}\label{thm0}
Let $\widehat{PE}: y=\varphi(x)$ be an increasing and concave smooth streamline. Suppose that the Mach number $M$ is an increasing along $\widehat{PE}$ with $M=1$ at point $P$. We further assume that $\varphi'$ and $M$ are $C^2$ functions satisfying the following condition
\begin{align}\label{1.5}
\fr{\varphi''}{1+(\varphi')^2}+\fr{2\sqrt{M^2-1}M'}{M(2+(\gamma-1) M^2)}<0,\quad {\rm on}\ \widehat{PE}.
\end{align}
Then there exists a smooth sonic curve $\widehat{PD}$ so that the Euler system \eqref{1.2} has a smooth solution in the region $PED$.
\end{theorem}
\begin{remark}\label{r0}
The condition \eqref{1.5} is obviously satisfied near the sonic point $P$. The detailed description of this condition can be found in Remark \ref{r2} in Section \ref{s2}.
\end{remark}

The rest of the paper is organized as follows. In Section \ref{s2}, we introduce the angle variables to formulate the problem and then state the main result. In Section \ref{s3}, we change the problem to a degenerate linear one and then establish the global existence of classical solutions to the linear problem. The properties of solutions up to degenerate line are also discussed in this section. Finally, we show that the coordinate transformation is global one-to-one, and then return the solution in the partial hodograph plane to that in the original physical plane in Section \ref{s4}.

\section{Reformulation of  problem and main result}\label{s2}

In this section, we first derive the system by introducing the flow angle and the Mach angle as the dependent variables, then detailedly reformulate the problem and restate the result in terms of these angle variables.

\subsection{Equations in terms of the angle variables}

The Euler system \eqref{1.2} can be written as
\begin{align}\label{2.1}
\left(\begin{array}{ccccc}
    c^2-u^2 & -uv   \\
       0    &-1
 \end{array}\right)
 \left(\begin{array}{c}
    u \\
    v
 \end{array}\right)_x+
\left(\begin{array}{ccccc}
    -uv & c^2-v^2 \\
    1 & 0
 \end{array}\right)
\left(\begin{array}{c}
    u \\
    v
 \end{array}\right)_y=0.
\end{align}
Standard calculation arrives at the characteristic form of \eqref{2.1}
\begin{align}\label{2.2}
\left\{
\begin{array}{l}
  \pa^+u+\Lambda_-\pa^+v =0,\\
  \pa^-u+\Lambda_+\pa^-v=0,
\end{array}
\right.
\end{align}
where $\pa^\pm$ are differential operators defined as $\pa^\pm=\pa_x+\Lambda_\pm\pa_y$, $\Lambda_\pm$ are positive and negative eigenvalues given in \eqref{1.4}.

As in \cite{LYZ, Li-Zheng1, ZhangT1}, we introduce the flow angle $\theta$ and the Mach angle $\omega$ as follows
\begin{align}\label{2.3}
\tan\theta=\fr{v}{u},\quad \sin\omega=\fr{1}{M},
\end{align}
and denote
\begin{align}\label{2.4}
\alpha:=\theta+\omega,\quad \beta:=\theta-\omega.
\end{align}
From the expressions of $\Lambda_\pm$ in \eqref{1.4}, one can see that
\begin{align}\label{2.5}
\tan\alpha=\Lambda_+,\quad \tan\beta=\Lambda_-.
\end{align}
Thus the angles $\alpha$ and $\beta$ are the inclination angles of characteristic curves.
Moreover, in terms of the angle variables, the velocity $(u,v)$ can be expressed as
\begin{align}\label{2.6}
u=c\fr{\cos\theta}{\sin\omega},\quad v=c\fr{\sin\theta}{\sin\omega}.
\end{align}

Denote
\begin{align}\label{2.7}
\bar{\pa}^+=\cos\alpha\pa_x+\sin\alpha\pa_y,\quad
\bar{\pa}^-=\cos\beta\pa_x+\sin\beta\pa_y,\quad
\bar{\pa}^0=\cos\theta\pa_x+\sin\theta\pa_y,
\end{align}
from which we have
\begin{align}\label{2.8}
\pa_x=-\fr{\sin\beta\bar{\pa}^+-\sin\alpha\bar{\pa}^-}{\sin(2\omega)},\quad \pa_y=\fr{\cos\beta\bar{\pa}^+-\cos\alpha\bar{\pa}^-}{\sin(2\omega)},
\quad
\bar{\pa}^0=\fr{\bar{\pa}^++\bar{\pa}^-}{2\cos\omega}.
\end{align}
Putting \eqref{2.6} into \eqref{1.3} leads to
\begin{align}\label{2.9}
\fr{\bar\pa^i c}{c}=\fr{\kappa}{\varpi(\kappa+\varpi^2)}\bar\pa^i\varpi,\quad i=\pm, 0,
\end{align}
where $\kappa=(\gamma-1)/2$ and $\varpi=\sin\omega$. Combining with \eqref{2.2} and \eqref{2.7}-\eqref{2.9}, we can obtain a new system in terms of the variables $(\theta,\omega)$ as follows
\begin{align}\label{2.10}
\left\{
\begin{array}{l}
   \bar{\pa}^+\theta+\fr{\cos\omega}{\kappa+\varpi^2}\bar{\pa}^+\varpi =0, \\[3pt]
   \bar{\pa}^-\theta-\fr{\cos\omega}{\kappa+\varpi^2}\bar{\pa}^-\varpi =0.
  \end{array}
\right.
\end{align}
Here and below, we apply the mixed variables $\omega$ and $\varpi$ in a system for convenience.

For later use, we further introduce a new variable
\begin{align}\label{2.11}
\Xi=\fr{1}{4\kappa}\ln\bigg(\fr{\varpi^2}{\kappa+\varpi^2}\bigg),
\end{align}
from which one has
\begin{align}\label{2.12}
\bar{\pa}^i\varpi=2\varpi(\kappa+\varpi^2)\bar{\pa}^i\Xi,\quad i=\pm, 0.
\end{align}
We insert the above into \eqref{2.10} to obtain
\begin{align}\label{2.13}
\left\{
\begin{array}{l}
   \bar{\pa}^+\theta+\sin(2\omega)\bar{\pa}^+\Xi =0, \\[3pt]
   \bar{\pa}^-\theta-\sin(2\omega)\bar{\pa}^-\Xi =0.
  \end{array}
\right.
\end{align}
According to the commutator relation between $\bar\pa^+$ and $\bar\pa^-$ \cite{LYZ, Li-Zheng1}
\begin{align}\label{2.14}
\bar{\pa}^-\bar{\pa}^+-\bar{\pa}^+\bar{\pa}^-
=&\fr{\cos(2\omega)\bar{\pa}^-\alpha-\bar{\pa}^+\beta}{\sin(2\omega)}\bar{\pa}^+
+\fr{\cos(2\omega)\bar{\pa}^+\beta-\tilde{\pa}^-\alpha}{\sin(2\omega)}\bar{\pa}^-,
\end{align}
we can derive the equations for $U:=\bar{\pa}^+\Xi$ and $V:=\bar{\pa}^-\Xi$
\begin{align}\label{2.15}
\left\{
\begin{array}{l}
\bar{\pa}^-U=\fr{\kappa U}{\cos^2\omega}(U-\cos(2\omega)V)+\fr{U}{\cos^2\omega}(U+V\cos^2(2\omega)), \\[3pt]
\bar{\pa}^+V=\fr{\kappa V}{\cos^2\omega}(V-\cos(2\omega)U)+\fr{V}{\cos^2\omega}(V+U\cos^2(2\omega)).
\end{array}
\right.
\end{align}

\subsection{Newly formulated problem and  main result}

We now set the problem in detail by mimicking the real setting of airfoil problem. Let $\widehat{PE}: y=\varphi(x)\ (x\in[x_1,x_2])$ be a smooth curve and vector function $(\hat{u}(x), \hat{v}(x))\ (x\in[x_1,x_2])$ be a velocity distribution on $\widehat{PE}$. Denote
\begin{align}\label{2.16}
\hat{c}(x)=\sqrt{\kappa[B_0-\hat{u}^2(x)-\hat{v}^2(x)]}.
\end{align}
Thanks to \eqref{2.3}, one can obtain the data of $(\theta, \varpi)$ on $\widehat{PE}$
\begin{align}\label{2.17}
\theta(x,\varphi(x))=\arctan\bigg(\fr{\hat{v}(x)}{\hat{u}(x)}\bigg)=:\hat{\theta}(x), \quad \varpi(x,\varphi(x))=\fr{\hat{c}(x)}{\sqrt{\hat{u}^2(x)+\hat{v}^2(x)}}=:\hat{\varpi}(x).
\end{align}
Our problem is reformulated as follows.
\begin{problem}\label{p1}
Let $\widehat{PE}: y=\varphi(x)\ (x\in[x_1,x_2])$ be a smooth curve. We assign the boundary data $(\theta, \varpi)=(\hat{\theta}, \hat{\varpi})(x)$ on $\widehat{PE}$ such that
\begin{align}\label{2.18}
\hat{\theta}(x)=\arctan\varphi'(x),  \hat{\varpi}(x)\in(0,1)\ \ \forall\ x\in(x_1, x_2], \ \ {\rm and}\ \hat{\varpi}(x_1)=1.
\end{align}
We find a smooth sonic curve $\widehat{PD}$ and build a smooth supersonic solution to system \eqref{2.10} in the region $PED$ with a negative characteristic curve $\widehat{ED}$. See Fig. 2.
\end{problem}
\begin{remark}\label{r1}
The conditions in \eqref{2.18} mean that the arc $\widehat{PE}$ is a streamline and the state of flow is supersonic on $\widehat{PE}$ with sonic at $P$.
\end{remark}

Assume that the functions $\varphi(x)$ and $\hat{\varpi}(x)$ satisfy
\begin{align}\label{2.19}
\begin{array}{l}
\varphi'(x), \hat{\varpi}(x)\in C^2([x_1,x_2]), \\
\varphi_0\leq\varphi'(x)\leq\varphi_1, \quad  \varphi''(x)<0, \quad \hat{\varpi}'(x)<0,
\end{array}
\end{align}
where $\varphi_0$ and $\varphi_1$ are some positive constants. Corresponding to \eqref{1.5}, we further suppose that $\varphi(x)$ and $\hat{\varpi}(x)$ satisfy
\begin{align}\label{2.20}
\bigg(\fr{\varphi''}{1+(\varphi')^2} -\fr{\sqrt{1-\hat{\varpi}^2}}{\kappa+\hat{\varpi}^2}\hat{\varpi}'\bigg)(x)<0,\quad \forall\ x\in[x_1,x_2].
\end{align}
\begin{remark}\label{r2}
It follows by \eqref{2.19} that the arc $\widehat{PE}$ is an increasing and concave streamline, along which the angle variable $\varpi$ is a decreasing function. We mention that this condition is reasonable in the``true" airfoil problem. The condition \eqref{2.20} is mainly used to determine the signs of $U$ and $V$ on the arc $\widehat{PE}$. This condition obviously holds near the sonic point $P$.
\end{remark}

Thus Theorem \ref{thm0} is restated in the next theorem.
\begin{theorem}\label{thm1}
Let \eqref{2.19} and \eqref{2.20} be satisfied.
Then there exists a smooth sonic curve $\widehat{PD}: y=\psi(x)\ (x\in[x_1,x_3])$ and Problem \ref{p1} admits a global smooth solution $(\theta, \varpi)\in C^2$ in the region $PED$. Moreover, the sonic curve $\widehat{PD}$ is $C^{1,\fr{1}{6}}$-continuous and the solution $(\theta, \varpi)(x,y)$ is uniformly $C^{1,\fr{1}{6}}$ up to $\widehat{PD}$. Furthermore, the function $\theta$ is strictly monotone decreasing along the sonic curve $\widehat{PD}$ and the negative characteristic $\widehat{DE}$.
\end{theorem}
\begin{remark}\label{r4}
The results in Theorem \ref{thm1} can also be applied to explore the properties of solutions in the region near the upstream vertex $P$ of supersonic bubble described as in Fig. 1.
\end{remark}

\subsection{The boundary information for $(U,V)$}

The strategy of this paper is to solve the singular system \eqref{2.15} for $(U,V)$ in a partial hodograph plane. Thus we need to derive the boundary data of $(U,V)$ on the arc $\widehat{PE}$ from the functions $(\hat{\theta}, \hat{\varpi})(x)$.

According to \eqref{2.8} and \eqref{2.10}, one gets
\begin{align}\label{2.21}
\dps\bar\pa^+\theta=\cos\omega\bar\pa^0\theta-\fr{\cos^2\omega}{\kappa+\varpi^2}\bar\pa^0\varpi, \quad
\dps\bar\pa^-\theta=\cos\omega\bar\pa^0\theta+\fr{\cos^2\omega}{\kappa+\varpi^2}\bar\pa^0\varpi,
\end{align}
which together with \eqref{2.13} acquires
\begin{align}\label{2.22}
U=\fr{\sqrt{1-\varpi^2}}{2\varpi(\kappa+\varpi^2)}\bar\pa^0\varpi-\fr{1}{2\varpi}\bar\pa^0\theta,\quad V=\fr{\sqrt{1-\varpi^2}}{2\varpi(\kappa+\varpi^2)}\bar\pa^0\varpi+\fr{1}{2\varpi}\bar\pa^0\theta.
\end{align}
Recalling the fact that the arc $\widehat{PE}$ is a streamline, we have
\begin{align}\label{2.22a}
\bar\pa^0\theta|_{\widehat{PE}}=\cos\hat{\theta}(x)\hat{\theta}'(x) =\fr{\cos\hat{\theta}(x)\varphi''(x)}{1+(\varphi'(x))^2},\quad \bar\pa^0\varpi|_{\widehat{PE}}=\cos\hat{\theta}(x)\hat{\varpi}'(x),
\end{align}
which combined with \eqref{2.22} yields
\begin{align}\label{2.23}
\begin{array}{l}
\dps U|_{\widehat{PE}}=\fr{\cos\hat{\theta}}{2\hat{\varpi}} \bigg(\fr{\sqrt{1-\hat{\varpi}^2}}{\kappa+\hat{\varpi}^2}\hat{\varpi}'-\fr{\varphi''}{1+(\varphi')^2}\bigg)(x) :=\hat{a}(x), \\[8pt]
\dps V|_{\widehat{PE}}=\fr{\cos\hat{\theta}}{2\hat{\varpi}} \bigg(\fr{\sqrt{1-\hat{\varpi}^2}}{\kappa+\hat{\varpi}^2}\hat{\varpi}'+\fr{\varphi''}{1+(\varphi')^2}\bigg)(x) :=\hat{b}(x).
\end{array}
\end{align}
For later use, we here give the boundary data $\bar\pa^0\Xi$ on $\widehat{PE}$. Combining with \eqref{2.12} and \eqref{2.22a} achieves
\begin{align}\label{2.23a}
\bar\pa^0\Xi|_{\widehat{PE}}=\fr{\cos\hat\theta(x)\hat{\varpi}'(x)}{2\hat{\varpi}(x)(\kappa+\hat{\varpi}^2(x))} :=\hat{d}(x).
\end{align}
Moreover, it suggests by the conditions \eqref{2.19} and \eqref{2.20} that
\begin{align}\label{2.24}
\begin{array}{c}
\hat{a}(x), \hat{b}(x), \hat{d}(x)\in C^1([x_1,x_2]), \\
\hat{m}_0\leq \hat{a}(x)\leq\hat{M}_0,\  -\hat{M}_0\leq\hat{b}(x),\hat{d}(x)\leq -\hat{m}_0,\ \ \forall\ x\in[x_1,x_2],
\end{array}
\end{align}
for some positive constants $\hat{m}_0$ and $\hat{M}_0$.

\section{Solutions in a partial hodograph plane}\label{s3}

In this section, we solve the singular system \eqref{2.15} with the boundary data \eqref{2.23} under the conditions \eqref{2.24} by introducing a partial hodograph plane.

\subsection{Reformulated problem in a  partial hodograph plane}

We reformulate the problem into a new linear problem by introducing a partial hodograph transformation. Introduce the coordinate transformation $(x,y)\mapsto(t,r)$
\begin{align}\label{3.1}
t=\cos\omega(x,y),\quad r=\hat{\theta}_1-\theta(x,y),
\end{align}
where $\hat{\theta}_1=\hat{\theta}(x_1)$,
from which and the conditions $\hat\theta'<0$ and $\hat\varpi'<0$. Then  we see that the arc $\widehat{PE}$
is transformed into a curve $\widehat{P'E'}:\ r=\tilde{r}(t)\ (t\in[0,t_0])$ in the half plane of $t\geq0$ defined through a parametric $x$
\begin{align}\label{3.2}
t=\cos\hat\omega(x),\ \ r=\hat{\theta}_1-\hat\theta(x), \quad (x\in[x_1,x_2]).
\end{align}
The number $t_0=\cos\hat\omega(x_2)$ is a positive constant. Moreover, the smooth function $r=\hat{\theta}_1-\hat\theta(x)$ is strictly increasing, which implies that there exists an inverse function, denoted by $x=\hat{x}(r)$ ($r\in[0,r_0]$), where $r_0=\tilde{r}(t_0)$. We now define $\hat{f}(r)=\hat{f}(\hat{x}(r))\ (f=a,b,d)$, and then obtain the boundary data of $(U,V)$ on $\widehat{P'E'}$
\begin{align}\label{3.3}
U|_{\widehat{P'E'}}=\hat{a}(r),\quad V|_{\widehat{P'E'}}=\hat{b}(r),\ \ \forall\ r\in[0,r_0].
\end{align}
It is obvious by \eqref{2.24} that
\begin{align}\label{3.4}
\begin{array}{c}
\hat{a}(r), \hat{b}(r), \hat{d}(r)\in C^1([0,r_0]), \\
\hat{m}_0\leq \hat{a}(r)\leq\hat{M}_0,\  -\hat{M}_0\leq\hat{b}(r), \hat{d}(r)\leq -\hat{m}_0,\ \ \forall\ r\in[0,r_0].
\end{array}
\end{align}

We now derive the equations of $(U,V)$ in terms of $(t,r)$. By a direct calculation, one finds by using \eqref{2.12}, \eqref{2.13} and \eqref{3.1} that
\begin{align}\label{3.5}
\bar{\pa}^+=-\fr{2F}{t}U\pa_t+2\sqrt{1-t^2}Ut\pa_r,\quad
\bar{\pa}^-=-\fr{2F}{t}V\pa_t-2\sqrt{1-t^2}Vt\pa_r,
\end{align}
where $F=F(t)=(1-t^2)(\kappa+1-t^2)>0$. Therefore, system \eqref{2.15} can be transformed to a new closed system of $(U(t,r),V(t,r))$
\begin{align}\label{3.6}
\left\{
\begin{array}{l}
\dps U_t+\fr{\sqrt{1-t^2}t^2}{F}U_r
=-\fr{(\kappa+1)U}{2FV} \fr{U+V}{t}+\fr{\kappa+2-2t^2}{F}Ut,\\[8pt]
\dps V_t-\fr{\sqrt{1-t^2}t^2}{F}V_r
=-\fr{(\kappa+1)V}{2FU} \fr{U+V}{t}+\fr{\kappa+2-2t^2}{F}Vt.
\end{array}
\right.
\end{align}
We further introduce
\begin{align}\label{3.7}
\overline{U}=\fr{1}{U},\quad \overline{V}=-\fr{1}{V},
\end{align}
to transform \eqref{3.6} to a linear system
\begin{align}\label{3.8}
\left\{
\begin{array}{l}
\dps \pa_+\overline{U}
=\fr{\overline{U}-\overline{V}}{2t}+\fr{\kappa+2-t^2}{2F}(\overline{U}-\overline{V})t -\fr{\kappa+2-2t^2}{F}\overline{U}t,\\[8pt]
\dps \pa_-\overline{V}
=\fr{\overline{V}-\overline{U}}{2t}-\fr{\kappa+2-t^2}{2F}(\overline{U}-\overline{V})t  -\fr{\kappa+2-2t^2}{F}\overline{V}t,
\end{array}
\right.
\end{align}
where
\begin{align}\label{3.8a}
\pa_+={\pa}_t+\lambda(t){\pa}_r,\quad \pa_-={\pa}_t-\lambda(t){\pa}_r,\quad \lambda(t)=\fr{\sqrt{1-t^2}t^2}{F(t)}.
\end{align}
In addition, combining with \eqref{3.3}, \eqref{3.4} and \eqref{3.7}, the boundary information of $(\overline{U}, \overline{V})$ on $\widehat{P'E'}$ are
\begin{align}\label{3.9}
\overline{U}_{\widehat{P'E'}}=\fr{1}{\hat{a}(r)}:=\bar{a}(r),\quad \overline{V}_{\widehat{P'E'}}=-\fr{1}{\hat{b}(r)}:=\bar{b}(r),\ \ \forall\ r\in[0,r_0],
\end{align}
with
\begin{align}\label{3.10}
\begin{array}{c}
\bar{a}(r), \bar{b}(r)\in C^1([0,r_0]), \\
\bar{m}_0\leq \bar{a}(r), \bar{b}(r)\leq\bar{M}_0,\ \ \forall\ r\in[0,r_0]
\end{array}
\end{align}
for two positive constants $\bar{m}_0$ and $\bar{M}_0$.

It is easy to see that the characteristic curves for system \eqref{3.8} are independent of the solution. Then from the point $E'(t_0, r_0)$, we draw the positive characteristic curve, denoted by $r=\check{r}(t)\ (t\in[0,t_0])$, up to the line $t=0$ at a point $D'(0, r^*)$, where the function $\check{r}(t)$ and the number $r^*$ are defined as
\begin{align}\label{3.11}
\check{r}(t)=r_0-\int_{t}^{t_0}\fr{\sqrt{1-s^2}s^2}{F(s)}\ {\rm d}s,\quad r^*=r_0-\int_{0}^{t_0}\fr{\sqrt{1-s^2}s^2}{F(s)}\ {\rm d}s.
\end{align}
Thus in terms of $(t,r)$-plane, our main result can be stated as follows.
\begin{theorem}\label{thm2}
Assume that \eqref{3.10} holds. Then there exists a global smooth solution $(\overline{U}(t,r),\overline{V}(t,r))$ for problem \eqref{3.8} \eqref{3.9} in the region $P'E'D'$. Moreover, the solution $(\overline{U},\overline{V})$ and the quantity $(\overline{U}-\overline{V})/t$ are uniformly $C^{\fr{1}{3}}$ up to $t=0$.
\end{theorem}

\subsection{Existence and uniform boundedness of solutions}

This subsection is devoted to proving Theorem \ref{thm2}. Let $\eps\in(0,t_0]$ be an arbitrary constant. Denote $\Omega_\eps=\{(t,r)|\ t\geq\eps\}\cap P'E'D'$.
Due to the facts that $\widehat{P'E'}$ is a space-like curve except point $P'$ and the linear system \eqref{3.8} is strictly hyperbolic in the region $\Omega_\eps$, the global existence of smooth solutions for problem \eqref{3.8} \eqref{3.9} can be directly obtained by the classical theory of linear equations. Thus we have
\begin{lemma}\label{lem1}
Under the assumption \eqref{3.10}, the linear problem \eqref{3.8} \eqref{3.9} admits a global solution $(\overline{U}(t,r),\overline{V}(t,r))\in C^1$ in the region $\Omega_\eps$.
\end{lemma}

We next establish the uniform boundedness of the solution $(\overline{U}(t,r),\overline{V}(t,r))$ in $\Omega_\eps$.
\begin{lemma}\label{lem2}
The solution $(\overline{U}(t,r),\overline{V}(t,r))$ in Lemma \ref{lem1} satisfies
\begin{align}\label{3.12}
\fr{1}{2}\bar{m}_0\leq \overline{U}(t,r), \overline{V}(t,r)\leq 2e^{k_0}\bar{M}_0,\ \ \forall\ (t,r)\in\Omega_\eps,
\end{align}
where $k_0=(\kappa+2)/(\kappa-\kappa t_{0}^2)$.
\end{lemma}

\begin{proof}
The proof consists of  two steps.

\textbf{Step I. The lower bound of $(\overline{U}(r,t),\overline{V}(r,t))$.} We claim that the solution $(\overline{U}(t,r),\overline{V}(t,r))$ satisfies
\begin{align}\label{3.14}
\overline{U}(t,r), \overline{V}(t,r)>\fr{1}{2}\bar{m}_0.
\end{align}
To show this, it is note that \eqref{3.14} holds in a small neighborhood of the point $E'$. We move the level set of  $t$ from $t=t_{E'}$ to $t=0$. Assume that the point $A$ is the first time  so that either 
$\overline{U}=\bar{m}_0/2$ or $\overline{V}=\bar{m}_0/2$ in the closed region bounded by $\widehat{E'P'}$, $\widehat{E'D'}$ and $t=t_A$. From the point $A$, we draw the negative and positive characteristic curves up to
the boundary $\widehat{P'E'}$ at points $B$ and $C$, respectively. We suppose without the loss of generality that $\overline{U}=\bar{m}_0/2$ at $A$ and then $\overline{U}>\bar{m}_0/2$, $\overline{V}>\bar{m}_0/2$ hold on $\widehat{AC}\setminus\{A\}$.
By the smoothness of $\overline{U}$, there should be $\pa_+\overline{U}\geq0$ at $A$. However, we check by the equation for $\overline{U}$ in \eqref{3.8} that
\begin{align*}
\pa_+\overline{U}|_{A}\leq -\fr{\kappa+2-2t_{A}^2}{F(t_A)}\cdot\fr{\bar{m}_0}{2}t_A<0,
\end{align*}
which yields a contradiction. Thus the inequalities in \eqref{3.14} are valid in $\Omega_\eps$ for any number $\eps>0$.

\textbf{Step II. The upper bound of $(\overline{U}(t,r),\overline{V}(t,r))$.} Denote
\begin{align}\label{3.15}
\widetilde{U}=e^{k_0t}\overline{U},\quad \widetilde{V}=e^{k_0t}\overline{V}.
\end{align}
Then the equations for $(\widetilde{U},\widetilde{V})$ are
\begin{align}\label{3.16}
\left\{
\begin{array}{l}
\dps \pa_+\widetilde{U}
=G(t,r) +\bigg(k_0-\fr{(\kappa+2-2t^2)t}{F}\bigg)\widetilde{U}:=H_1(t,r),\\[8pt]
\dps \pa_-\widetilde{V}
=-G(t,r) +  \bigg(k_0-\fr{(\kappa+2-2t^2)t}{F}\bigg)\widetilde{V}:=H_2(t,r),
\end{array}
\right.
\end{align}
where
$$
G(t,r)=\bigg(\fr{1}{2t}+\fr{(\kappa+2-t^2)t}{2F}\bigg)(\widetilde{U}-\widetilde{V}).
$$
It follows from the chosen of $k_0$ that
$$
k_0-\fr{(\kappa+2-2t^2)t}{F}>0.
$$
We now establish the upper bound of $(\widetilde{U},\widetilde{V})$. For a point $A$ in $\Omega_\eps$, let's assume $\widetilde{U}\geq\widetilde{V}$ at $A$. Otherwise the proof can be derived symmetrically. From $A$, we draw the positive characteristic curve up to the boundary $\widehat{P'E'}$ at point $C$. See Fig. 3. We divide the proof into two cases.

\noindent \textbf{Case 1. $H_1(t,r)\geq0$ always holds on $\widehat{AC}$.} For this case, we observe that
$$
\pa_+\widetilde{U}\geq0,
$$
is valid on $\widehat{AC}$, which implies that $\widetilde{U}$ is an increasing function. Therefore, we obtain
\begin{align}\label{3.17}
\widetilde{V}|_A\leq\widetilde{U}|_A\leq\widetilde{U}|_C= e^{k_0t_C}\overline{U}|_C\leq e^{k_0}\bar{M}_0.
\end{align}

\begin{figure}[htbp]
\begin{center}
\includegraphics[scale=0.5]{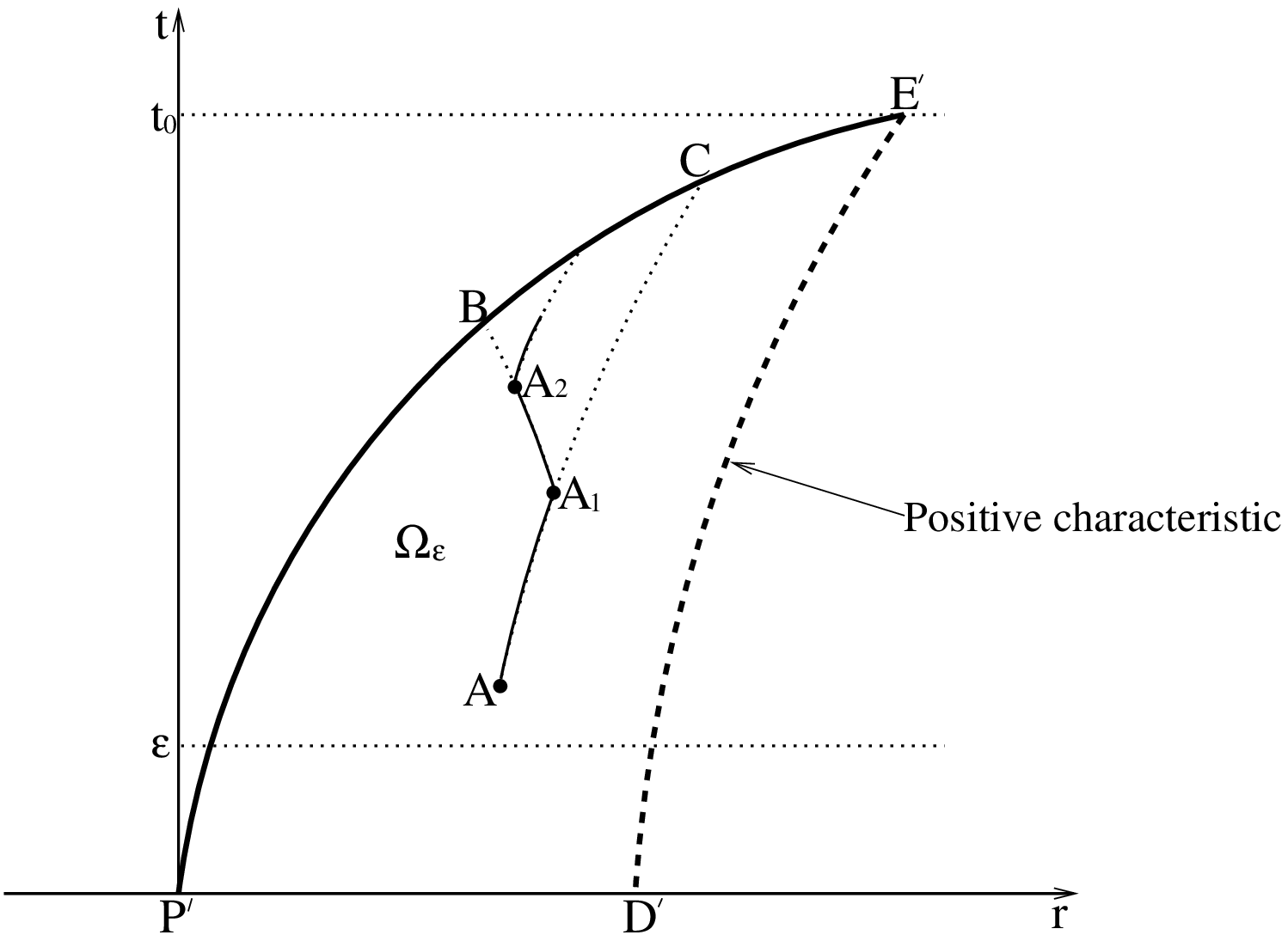}
\caption{\footnotesize The region $\Omega_\eps$.}
\end{center}
\end{figure}

\noindent \textbf{Case 2. $H_1(t,r)<0$ holds at some points on $\widehat{AC}$.} Note that $H_1>0$ at $A$. We can choose a point $A_1$ on $\widehat{AC}$ such that $H_1|_{A_1}=0$ and $H_1>0$ on $\widehat{AA_1}$. Then the function $\widetilde{U}$ is increasing along $\widehat{AA_1}$, from which one gets
\begin{align}\label{3.18}
\widetilde{V}|_A\leq\widetilde{U}|_A\leq\widetilde{U}|_{A_1}.
\end{align}
Due to $H_1=0$ at point $A_1$, we find that $G|_{A_1}<0$ by Step I and then $\widetilde{U}|_{A_1}<\widetilde{V}|_{A_1}$, which together with \eqref{3.18} gives
\begin{align}\label{3.19}
\widetilde{V}|_A\leq\widetilde{U}|_A\leq\widetilde{U}|_{A_1}<\widetilde{V}|_{A_1}.
\end{align}
Moreover, it is easily seen that $H_2>0$ at point $A_1$. We now draw the negative characteristic curve from $A_1$ up to the boundary $\widehat{P'E'}$ at point $B$, as in Fig. 3. If $H_2\geq0$ always holds on the curve $\widehat{A_1B}$, one has $\pa_-\widetilde{V}>0$ which indicates that $\widetilde{V}$ is an increasing function along $\widehat{A_1B}$ and then there holds $\widetilde{V}_{A_1}\leq\widetilde{V}_{B}$. Combining the above and \eqref{3.19} leads to
\begin{align}\label{3.20}
\widetilde{V}|_A\leq\widetilde{U}|_A\leq\widetilde{U}|_{A_1} <\widetilde{V}|_{A_1}\leq\widetilde{V}_{B}=e^{k_0t_{B}}\overline{V}_{B}\leq e^{k_0}\bar{M}_0.
\end{align}
If there exists a point at which $H_2<0$, then in view of the fact $H_2>0$ at point $A_1$ we can choose a point $A_2$ on $\widehat{A_1B}$ such that $H_2=0$ at $A_2$ and $H_2\geq0$ on $\widehat{A_1A_2}$. Now the function $\widetilde{V}$ is increasing along $\widehat{A_1A_2}$ and then
$$
\widetilde{V}|_A\leq\widetilde{U}|_A\leq\widetilde{U}|_{A_1} <\widetilde{V}|_{A_1}\leq\widetilde{V}_{A_2},
$$
which, combined with the fact $(\widetilde{U}-\widetilde{V})|_{A_2}>0$ following from $H_2=0$ at $A_2$,   yields
$$
\widetilde{V}|_A\leq\widetilde{U}|_A\leq\widetilde{U}|_{A_1} <\widetilde{V}|_{A_1}\leq\widetilde{V}_{A_2}<\widetilde{U}_{A_2}.
$$
Furthermore, we notice that $H_1>0$ at point $A_2$. Then we can repeat the above process as point $A_1$ to obtain
\begin{align}\label{3.20a}
\widetilde{U}|_A, \widetilde{V}|_A\leq 2e^{k_0}\bar{M}_0.
\end{align}
In fact, the above repeating process includes two subcases: Subcase (I) the process can be completed in finite steps; Subcase (II) the process can not be completed in finite steps. For the Subcase (I), it is easily seen that \eqref{3.20a} holds. For the Subcase (II), by the construction, there exists a sequence of points $(A_k)_{k=1}^\infty$ such that $H_1>0, H_2=0$ or $H_1=0, H_2>0$ hold at these points. Since the region is bounded, there exists an accumulation point $A_\infty$ such that $H_1\geq0, H_2\geq0$ at the point $A_\infty$. We claim that the point $A_\infty$ must be on the boundary $\widehat{P'E'}$. If not, then it follows by the continuity that $H_1(A_\infty)=H_2(A_\infty)=0$, from which and the expressions of $H_1, H_2$ in \eqref{3.16} one has $\widetilde{U}(A_\infty)=\widetilde{V}(A_\infty)=0$ and then $\overline{U}(A_\infty)=\overline{V}(A_\infty)=0$ by \eqref{3.15} which contradicts to \eqref{3.14}. Since $A_\infty$ is on $\widehat{P'E'}$, then $\widetilde{U}(A_\infty),\widetilde{V}(A_\infty)\leq e^{k_0}\bar{M}_0$ by the boundary conditions in \eqref{3.10}. Thanks to the continuity, there exists a small number $\delta>0$ such that $\widetilde{U},\widetilde{V}\leq 2e^{k_0}\bar{M}_0$ for any point in the region $O(A_\infty, \delta)\cap P'E'D'$, where $O(A_\infty, \delta)$ represents the $\delta$-neighborhood of point $A_\infty$. It is clear that there exists a positive integer $N$ such that the point $A_N$ is in the region $O(A_\infty, \delta)\cap P'E'D'$. Hence we have
\begin{align*}
\widetilde{U}|_A, \widetilde{V}|_A\leq \widetilde{U}|_{A_N}\leq 2e^{k_0}\bar{M}_0, \ \ {\rm or}\ \ \widetilde{U}|_A, \widetilde{V}|_A\leq \widetilde{V}|_{A_N}\leq 2e^{k_0}\bar{M}_0,
\end{align*}
which lead to \eqref{3.20a}.

Combining with \eqref{3.15} and \eqref{3.20a} arrives at
$$
\overline{U}|_A, \overline{V}|_A\leq e^{-k_0t_{A}}\cdot 2e^{k_0}\bar{M}_0\leq 2e^{k_0}\bar{M}_0,
$$
which completes the proof the lemma.
\end{proof}

\subsection{Properties of solutions}

We discuss in this subsection the properties of solutions near the degenerate line $\widehat{P'D'}$.

It is noted that the right-hand terms in \eqref{3.8} include a singular term $\overline{W}:=(\overline{U}-\overline{V})/t$. In order to extend the solution up to $t=0$, we need to establish the uniform boundedness of $\overline{W}$. We first derive the boundary data $\overline{W}$ on $\widehat{P'E'}$. According to the definitions of $(U,V)$, we find that the term $(U+V)/(2t)$ in the $t$-$r$ plane corresponds to the term $\bar\pa^0\Xi$ in the $x$-$y$ plane. Thus one has,  in light of  \eqref{2.23a}, \eqref{3.3} and \eqref{3.7},
\begin{align}\label{3.21}
\overline{W}|_{\widehat{P'E'}}=\fr{2}{UV}\bigg|_{\widehat{P'E'}}\cdot\fr{U+V}{2t}\bigg|_{\widehat{P'E'}} =\fr{2\hat{d}(r)}{\hat{a}(r)\hat{b}(r)},
\end{align}
which together with \eqref{3.4} acquires that $\overline{W}$ is uniformly bounded on the boundary curve $\widehat{P'E'}$. We next show that $\overline{W}$ is also uniformly bounded in the region $P'E'D'$ including the line $t=0$. Making use of \eqref{3.8}, one can easily obtain the equations for $\overline{W}$
\begin{align}\label{3.22}
\dps\pa_+\overline{W}=\fr{t^2}{F}(\overline{U}-\overline{V})-\fr{2\sqrt{1-t^2}}{F}t\overline{V}_r.
\end{align}
It can be seen from \eqref{3.22} that in order to acquire the boundedness of $\overline{W}$, we only need to estimate the term $t\overline{V}_r$. We comment that the terms $\overline{U}_r$ and $\overline{V}_r$ may not be uniformly bounded due to the degeneracy. Denote
\begin{align}\label{3.23}
R=t\overline{U}_r, \quad S=t\overline{V}_r.
\end{align}
Then by \eqref{3.8} the functions $(R, S)$  satisfy
\begin{align}\label{3.24}
\left\{
\begin{array}{l}
\dps \pa_+R
=\fr{3}{2}\fr{R}{t}-\bigg(\fr{1}{2t}+\fr{(\kappa+2-t^2)t}{2F}\bigg)S+\fr{(3t^2-\kappa-2)t}{2F}R, \\[8pt]
\dps \pa_-S
=\fr{3}{2}\fr{S}{t}-\bigg(\fr{1}{2t}+\fr{(\kappa+2-t^2)t}{2F}\bigg)R+\fr{(3t^2-\kappa-2)t}{2F}S,
\end{array}
\right.
\end{align}
from which one has
\begin{align}\label{3.25}
\left\{
\begin{array}{l}
\dps \pa_+(t^{-\fr{3}{2}}R)
=-\bigg(\fr{1}{2}+\fr{(\kappa+2-t^2)t^2}{2F}\bigg)t^{-\fr{5}{2}}S+\fr{(3t^2-\kappa-2)t^2}{2F}t^{-\fr{5}{2}}R, \\[8pt]
\dps \pa_-(t^{-\fr{3}{2}}S)
=-\bigg(\fr{1}{2}+\fr{(\kappa+2-t^2)t^2}{2F}\bigg)t^{-\fr{5}{2}}R+\fr{(3t^2-\kappa-2)t^2}{2F}t^{-\fr{5}{2}}S.
\end{array}
\right.
\end{align}
Set
$$
\eps_0=\min\bigg\{t_0,\ \fr{1}{4k_0}\bigg\}.
$$
Then the region $P'E'D'$ is divided into two parts $\Omega_1:=P'E'D'\cap\{t\leq\eps_0\}$ and $\Omega_2:=P'E'D'\cap\{t\geq\eps_0\}$. It is obvious that we just only need to show the boundedness of $\overline{W}$ in $\Omega_1$. By the chosen of $\eps_0$, we know that for $t\leq\eps_0$
\begin{align}\label{3.26}
\fr{1}{2}+\bigg|\fr{(\kappa+2-t^2)t^2}{2F}\bigg|\leq \fr{5}{8},\quad \bigg|\fr{(\kappa+2-3t^2)t^2}{2F}\bigg|\leq\fr{1}{8}.
\end{align}
Let
\begin{align}\label{3.27}
\overline{M}=1+2\max\bigg\{\max_{\Omega_2}\big\{|R|, |S|\big\},\ \max_{\widehat{P'E'}}\big\{|R|, |S|\big\}\bigg\}.
\end{align}
The term $\max_{\widehat{P'E'}}\{|R|, |S|\}$ is bounded by the assumptions in \eqref{3.10}.
Hence $\overline{M}$ is a uniformly bounded constant. We have the following lemma.

\begin{lemma}\label{lem3}
For any point $(t,r)\in\Omega_1$, there hold
\begin{align}\label{3.28}
|R|< \overline{M}, \quad |S|< \overline{M}.
\end{align}
\end{lemma}
\begin{proof}
The lemma is also proved by the contradiction argument. We note that \eqref{3.28} is valid in a small neighborhood of the segment $P'E'D'\cap\{t=\eps_0\}$ and then move the level set of  $t$ from $t=t_0$ to $t=0$. Suppose that the point $A$ is the first time so that either $|R(t_A, r_A)|=\overline{M}$ or  $|S(t_A, r_A)|=\overline{M}$ in the closed region bounded by $t=t_0$, $\widehat{P'E'}$, $\widehat{D'E'}$ and $t=t_A$. From the point $A$, we draw the negative and positive characteristic curves up to the boundary of $\Omega_1$ at points $B(t_B, r_B)$ and $C(t_C, r_C)$, respectively.
Without the loss of generality, we assume that $|R(t_A, r_A)|=\overline{M}$ and $|R(t,r)|\leq\overline{M}$, $|S(t,r)|\leq\overline{M}$ hold on the positive characteristic curve $\widehat{AC}$.
According to the definition of $\overline{M}$, we obtain that
$$
|R(t_B, r_B)|,\ |S(t_B, r_B)|,\ |R(t_C, r_C)|,\ |S(t_C, r_C)|<\overline{M}.
$$
Then we integrate the equation for $R$ in \eqref{3.25} from $A$ to $C$ to get
\begin{align*}
&t_{C}^{-\fr{3}{2}}R(t_C, r_C)-t_{A}^{-\fr{3}{2}}R(t_A, r_A)  \\
=&\int_{t_A}^{t_C}\bigg\{-\bigg(\fr{1}{2}+\fr{(\kappa+2-t^2)t^2}{2F}\bigg)t^{-\fr{5}{2}}S +\fr{(3t^2-\kappa-2)t^2}{2F}t^{-\fr{5}{2}}R\bigg\}\ {\rm d}t,
\end{align*}
from which and \eqref{3.26} we further obtain
\begin{align}\label{3.29}
|R(t_A, r_A)|\leq & t_{A}^{\fr{3}{2}}\bigg\{ t_{C}^{-\fr{3}{2}}|R(t_C, r_C)| + \int_{t_A}^{t_C}\bigg(\fr{5}{8}t^{-\fr{5}{2}}|S| +\fr{1}{8}t^{-\fr{5}{2}}|R|\bigg)\ {\rm d}t\bigg\} \nonumber \\
\leq & t_{A}^{\fr{3}{2}}\bigg\{ t_{C}^{-\fr{3}{2}}|R(t_C, r_C)| + \int_{t_A}^{t_C}\bigg(\fr{5}{8}t^{-\fr{5}{2}}\overline{M} +\fr{1}{8}t^{-\fr{5}{2}}\overline{M}\bigg)\ {\rm d}t\bigg\} \nonumber \\
\leq & t_{A}^{\fr{3}{2}}\bigg\{ t_{C}^{-\fr{3}{2}}|R(t_C, r_C)| + \fr{3}{4}\overline{M}\int_{t_A}^{t_C}t^{-\fr{5}{2}}\ {\rm d}t\bigg\} \nonumber \\
\leq & t_{A}^{\fr{3}{2}}\bigg\{ t_{C}^{-\fr{3}{2}}\bigg(|R(t_C, r_C)|-\fr{1}{2}\overline{M}\bigg) + \fr{1}{2}\overline{M}t_{A}^{-\fr{3}{2}}\bigg\}\leq \fr{1}{2}\overline{M}<\overline{M}.
\end{align}
This  contradicts to the assumption $|R(t_A, r_A)|=\overline{M}$. Here we have used the fact
$$
|R(t_C, r_C)|-\fr{1}{2}\overline{M}\leq0,
$$
which follows from the chosen of $\overline{M}$ in \eqref{3.27}. The proof of the lemma is completed.
\end{proof}

Based on Lemma \ref{lem3} and \eqref{3.22}, we have the following lemma.
\begin{lemma}\label{lem4}
The function $\overline{W}$ is uniformly bounded up to the degenerate line $\widehat{P'D'}$.
\end{lemma}
\begin{proof}
For any point $A(t_A, r_A)$ in the region $P'E'D'$, we draw the positive characteristic curve up to the boundary curve $\widehat{P'E'}$ at point $C(t_C, r_C)$. Integrating the first equation of \eqref{3.22} from $A$ to $C$ gives
\begin{align}\label{3.31a}
\overline{W}(t_A, r_A)=\overline{W}( t_C, r_C) -\int_{t_A}^{t_C}\bigg\{\fr{t^2}{F}(\overline{U}-\overline{V})-\fr{2\sqrt{1-t^2}}{F}t\overline{V}_r\bigg\}\ {\rm d}t,
\end{align}
which along with \eqref{3.12}, \eqref{3.21} and \eqref{3.28} leads to
\begin{align}\label{3.31}
|\overline{W}(t_A, r_A)|\leq & |\overline{W}( t_C, r_C)| +\int_{t_A}^{t_C}\bigg\{\fr{t^2}{F}(|\overline{U}|+|\overline{V}|) +\fr{2\sqrt{1-t^2}}{F}|t\overline{V}_r|\bigg\}\ {\rm d}t \nonumber \\
\leq &\fr{2\hat{c}(r_C)}{\hat{a}(r_C)\hat{b}(r_C)} +\fr{2k_0t_0}{\kappa+2}(2e^{k_0}\bar{M}_0+\overline{M})<\infty.
\end{align}
This implies, due to the arbitrariness of point $A$,  that $\overline{W}$ is uniformly bounded. The proof of the lemma is complete.
\end{proof}

By the uniform boundedness of $\overline{W}$, we see that $\overline{U}=\overline{V}$ holds on the degenerate line $\widehat{P'D'}$. Making use of this fact, we now establish the uniform continuity of $\overline{U}, \overline{V}$ and $\overline{W}$.
\begin{lemma}\label{lem5}
The functions $\overline{U}, \overline{V}$ and $\overline{W}$ are uniformly $C^{\fr{1}{3}}$ continuous in the whole domain $P'E'D'$, including the degenerate line $\widehat{P'D'}$.
\end{lemma}
\begin{proof}
We first show the the uniform continuity of $\overline{U}$.
For any two points $(0, r_1)$ and $(0, r_2)$ with $0<r_1<r_2\leq r_0$, we directly calculate
\begin{align}\label{3.32}
&\overline{U}(0, r_2)-\overline{U}(0, r_1)=\overline{V}(0, r_2)-\overline{U}(0, r_1) \nonumber \\
=&\overline{V}(t_m, r_m)-\overline{U}(t_m, r_m) \nonumber \\
&\ + \int_{0}^{t_m} \bigg\{ -\fr{\overline{W}}{2}-\fr{\kappa+2-t^2}{2F}(\overline{U}-\overline{V})t  -\fr{\kappa+2-2t^2}{F}\overline{V}t\bigg\}(t, r_-(t))\ {\rm d}t\nonumber \\[3pt]
&\ \ +\int_{0}^{t_m} \bigg\{ \fr{\overline{W}}{2}+\fr{\kappa+2-t^2}{2F}(\overline{U}-\overline{V})t -\fr{\kappa+2-2t^2}{F}\overline{U}t \bigg\}(t, r_+(t))\ {\rm d}t,
\end{align}
where $(t_m, r_m)$ and $r_\pm(t)$ are defined as
\begin{align}\label{3.33}
&\quad \ \  r_m=r_1+\int_{0}^{t_m}\lambda(t)\ {\rm d}t=r_2-\int_{0}^{t_m}\lambda(t)\ {\rm d}t,\\
&r_+(t)=r_1+\int_{0}^t\lambda(t)\ {\rm d}t,\quad r_-(t)=r_2-\int_{0}^t\lambda(t)\ {\rm d}t. \nonumber
\end{align}
With the aid of Lemmas \ref{lem2} and \ref{lem4}, it follows from \eqref{3.32} that
\begin{align}\label{3.34}
|\overline{U}(0, r_2)-\overline{U}(0, r_1)|\leq \widetilde{M}t_m
\end{align}
for some constant $\widetilde{M}>0$. Recalling the expression of $\lambda$ in \eqref{3.8a} yields
\begin{align}\label{3.34c}
\underline{k}t^2=:\fr{1}{\kappa+1}t^2\leq \lambda(t)\leq \fr{1}{\kappa\sqrt{1-t_{0}^2}}t^2=:\overline{k}t^2,
\end{align}
which together with \eqref{3.33} arrives at
\begin{align}\label{3.35}
\fr{2}{3}\underline{k}t_{m}^3\leq|r_2-r_1|\leq \fr{2}{3}\overline{k}t_{m}^3.
\end{align}
Combining with \eqref{3.34} and \eqref{3.35}, we obtain for some constant $M>0$
\begin{align}\label{3.36}
|\overline{U}(0, r_2)-\overline{U}(0, r_1)|\leq M|r_2-r_1|^{\fr{1}{3}}.
\end{align}

For any two points $(t_1, r_1)$ and $(0, r_2)$ with $t_1\leq t_0$ and $0<r_1<r_2\leq r_0$, we have
\begin{align}\label{3.37}
|\overline{U}(0, r_2)-\overline{U}(t_1, r_1)|\leq &|\overline{U}(0, r_2)-\overline{U}(0, r_1)| +|\overline{U}(0, r_1)-\overline{U}(t_1, r_1)| \nonumber \\
\leq & M|r_2-r_1|^{\fr{1}{3}} +Mt_1\leq M|(0, r_2)-(t_1, r_1)|^{\fr{1}{3}}
\end{align}
for some constant $M>0$. Here we used \eqref{3.36} and the fact that $\overline{U}_t$ is uniformly bounded. The uniform boundedness of $\overline{U}_t$ follows from \eqref{3.8}, \eqref{3.34c} and Lemma \ref{lem4}.
For any two points $(t_1, r_1)$ and $(t_2, r_2)$ in $P'E'D'$, we can obtain the similar inequality by inserting some suitable points and using the mean value theorem. Hence, the function $\overline{U}(t,r)$ is uniformly $C^{\fr{1}{3}}$ continuous in the whole domain $P'E'D'$. The uniform H\"older continuity of $\overline{V}(t,r)$ can be established in a similar way.

For the function $\overline{W}(t,r)$, we let $r_1, r_2, r_m$ as before and denote
$$
\tilde{r}_+(t)=r_2+\int_{0}^{t}\lambda(t)\ {\rm d}t,\quad \tilde{r}_m=\tilde{r}_+(t_m),
$$
which along with \eqref{3.33} gives $\tilde{r}_m-r_m=r_2-r_1$.
Moreover, recalling \eqref{3.22} finds that
\begin{align}\label{3.38}
\begin{array}{l}
\dps \overline{W}(0,r_1)=\overline{W}(t_m,r_m)-\int_{0}^{t_m} \bigg\{\fr{t^2}{F}(\overline{U}-\overline{V})-\fr{2\sqrt{1-t^2}}{F}S\bigg\}(t,r_+(t))\ {\rm d}t,  \\[8pt] \dps \overline{W}(0,r_2)=\overline{W}(t_m,\tilde{r}_m)-\int_{0}^{t_m} \bigg\{\fr{t^2}{F}(\overline{U}-\overline{V})-\fr{2\sqrt{1-t^2}}{F}S\bigg\}(t,\tilde{r}_+(t))\ {\rm d}t,
\end{array}
\end{align}
which are well-defined by Lemma \ref{lem3}. In view of the definition of $\overline{W}=(\overline{U}-\overline{V})/t$, Lemma \ref{lem3} and the inequality \eqref{3.35}, we obtain
\begin{align}\label{3.39}
|\overline{W}(t_m,r_m)- \overline{W}(t_m,\tilde{r}_m)|\leq&\fr{|\overline{U}(t_m,r_m)- \overline{U}(t_m,\tilde{r}_m)| +|\overline{V}(t_m,r_m)- \overline{V}(t_m,\tilde{r}_m)|}{t_m} \nonumber \\
\leq & \fr{2\overline{M}|r_m-\tilde{r}_m|}{t_{m}^2}=\fr{2\overline{M}|r_2-r_1|}{t_{m}^2} \leq \fr{4}{3}\overline{k}\cdot\overline{M}t_m.
\end{align}
Combining with \eqref{3.39} and \eqref{3.38} and applying Lemma \ref{lem3} again, one gets
\begin{align*}
&|\overline{W}(0,r_1)- \overline{W}(0,r_2)|\\
\leq&|\overline{W}(t_m,r_m)- \overline{W}(t_m,\tilde{r}_m)| +2\int_{0}^{t_m} \bigg\{\fr{t^2}{F}(|\overline{U}|+|\overline{V}|)+\fr{2\sqrt{1-t^2}}{F}|S|\bigg\}\ {\rm d}t  \\
\leq&\widetilde{M}t_m,
\end{align*}
for some constant $\widetilde{M}>0$. Repeating the same process as \eqref{3.34}-\eqref{3.37}, we acquire the  uniform $C^{\fr{1}{3}}$-continuity of $\overline{W}(t,r)$ in the whole domain $P'E'D'$. The proof of the lemma is finished.
\end{proof}

\section{Solutions in the physical plane}\label{s4}

In this section, we recover a global smooth supersonic-sonic solution to system \eqref{2.10} by expressing the solution in the partial hodograph plane back to that in the original physical plane.

\subsection{Inversion}

Thanks to \eqref{3.7} and \eqref{3.12}, we obtain the smooth solution $(U(t,r),V(t,r))$ of system \eqref{3.6} in the whole region $P'E'D'$. Now we pursue to construct the functions $x(t,r)$ and $y(t,r)$. Recalling the coordinate transformation \eqref{3.1} yields
$$
1=-\sin\omega(\omega_xx_t+\omega_yy_t),\quad 0=-\theta_xx_t-\theta_yy_t,
$$
from which one has
$$
x_t=\fr{\theta_y}{\sin\omega\omega_y\theta_x-\sin\omega\omega_x\theta_y},\quad y_t=-\fr{\theta_x}{\sin\omega\omega_y\theta_x-\sin\omega\omega_x\theta_y}.
$$
Similarly, it follows that
$$
x_r=\fr{\omega_y}{\omega_y\theta_x-\omega_x\theta_y},\quad y_r=-\fr{\omega_x}{\omega_y\theta_x-\omega_x\theta_y}.
$$
Making use of \eqref{2.8} and \eqref{2.13} gives
\begin{align}\label{4.3}
J=:\sin\omega\omega_y\theta_x-\sin\omega\omega_x\theta_y &=-\sin\omega(\bar{\pa}^+\omega\bar{\pa}^-\Xi+\bar{\pa}^-\omega\bar{\pa}^+\Xi) \nonumber \\
&=\fr{-4F(t)}{t}U(t,r)V(t,r).
\end{align}
Thus we have
\begin{align}\label{4.1}
\fr{\pa x}{\pa t}=\fr{\theta_y}{J},\quad
\fr{\pa y}{\pa t}=-\fr{\theta_x}{J},\quad
\fr{\pa x}{\pa r}=\fr{\sin\omega\omega_y}{J},\quad
\fr{\pa y}{\pa r}=-\fr{\sin\omega\omega_x}{J},
\end{align}
where the functions $(\theta_x, \theta_y, \omega_x, \omega_y)$ in terms of $(t,r)$ by \eqref{2.8}, \eqref{2.12} and \eqref{2.13} are
\begin{align}\label{4.2}
\begin{array}{l}
\theta_x=F_1(t,r)U(t,r)+F_2(t,r)V(t,r), \\
\theta_y=-F_3(t,r)U(t,r)-F_4(t,r)V(t,r), \\[4pt]
\dps\omega_x=-\fr{\kappa+1-t^2}{t^2}[F_1(t,r)U(t,r) -F_2(t,r)V(t,r)],\\[4pt]
\dps\omega_y= \fr{\kappa+1-t^2}{t^2}[F_3(t,r)U(t,r) -F_4(t,r)V(t,r)].
\end{array}
\end{align}
Here the expressions of functions $F_i(t,r)\ (i=1,2,3,4)$ are
\begin{align*}
F_1(t,r)&=t\sin(\hat{\theta}_1-r)-\sqrt{1-t^2}\cos(\hat{\theta}_1-r), \\
F_2(t,r)&=t\sin(\hat{\theta}_1-r) +\sqrt{1-t^2}\cos (\hat{\theta}_1-r), \\
F_3(t,r)&=t\cos(\hat{\theta}_1-r)+\sqrt{1-t^2}\sin(\hat{\theta}_1-r), \\
F_4(t,r)&=t\cos(\hat{\theta}_1-r) -\sqrt{1-t^2}\sin (\hat{\theta}_1-r).
\end{align*}
Thus we get by \eqref{4.1}-\eqref{4.2}
\begin{align}\label{4.4}
x_t=\dps\fr{F_3(t,r)U(t,r)+F_4(t,r)V(t,r)}{4F(t)U(t,r)V(t,r)}t, \quad
y_t=\dps\fr{F_1(t,r)U(t,r)+F_2(t,r)V(t,r)}{4F(t)U(t,r)V(t,r)}t,
\end{align}
and
\begin{align}\label{4.5}
x_r=\dps\fr{-F_3(t,r)U(t,r)+F_4(t,r)V(t,r)}{4t\sqrt{1-t^2}U(t,r)V(t,r)},  \quad
y_r=\dps\fr{-F_1(t,r)U(t,r)+F_2(t,r)V(t,r)}{4t\sqrt{1-t^2}U(t,r)V(t,r)}.
\end{align}
It follows from \eqref{4.4} and \eqref{4.5} that
\begin{align}\label{4.6}
\dps\fr{{\rm d}x(t, r_+(t))}{{\rm d}t}=\fr{F_4(t,r)}{2F(t)U(t,r)}t,\quad
\dps\fr{{\rm d}y(t, r_+(t))}{{\rm d}t}=\fr{F_2(t,r)}{2F(t)U(t,r)}t,
\end{align}
where the function $r_+(t)$ is defined by
$$
\fr{{\rm d}r_+(t)}{{\rm d}t}=\fr{\sqrt{1-t^2}t^2}{F(t)},
$$
the positive characteristics curve of system \eqref{3.6}.

From any point $(\hat{t},\hat{r})$ in the region $P'E'D'$, we draw the positive characteristic curve $r=\hat{r}_+(t)$ up to the boundary $\widehat{P'E'}$ at a unique point $(\bar{t}, \tilde{r}(\bar{t}))$ defined by
$$
\hat{r}_+(t)=\hat{r}+ \int_{\hat{t}}^{t}\fr{\sqrt{1-s^2}s^2}{F(s)}\ {\rm d}s,\ \ (t\geq \hat{t}), \quad \hat{r}_+(\bar{t})=\tilde{r}(\bar{t}).
$$
Note that the point $(\bar{t}, \tilde{r}(\bar{t}))$ is corresponding to the point $(\hat{x}(\tilde{r}(\bar{t})), \varphi(\hat{x}(\tilde{r}(\bar{t}))))$ in the $x$-$y$ plane by the definition of function $\hat{x}(r)$. We now apply \eqref{4.6} to define the values $x(\hat{t}, \hat{r})$ and $y(\hat{t}, \hat{r})$ as follows
\begin{align}\label{4.7}
\begin{array}{l}
x(\hat{t}, \hat{r})=\hat{x}(\tilde{r}(\bar{t})) -\dps\int_{\hat{t}}^{\bar{t}}\fr{F_4(t,\hat{r}_+(t))}{2F(t)U(t,\hat{r}_+(t))}t\ {\rm d}t,\\[10pt]
y(\hat{t}, \hat{r})=\varphi(\hat{x}(\tilde{r}(\bar{t}))) -\dps\int_{\hat{t}}^{\bar{t}}\fr{F_2(t,\hat{r}_+(t))}{2F(t)U(t,\hat{r}_+(t))}t\ {\rm d}t.
\end{array}
\end{align}
According to the arbitrariness of $(\hat{t},\hat{r})$, we obtain by \eqref{4.7} the functions $x=x(t,r)$ and $y=y(t,r)$ defined on the whole region $P'E'D'$.

Standard calculation provides the Jacobian of the mapping $(t,r)\mapsto(x,y)$
$$
j:=\fr{\pa(x,y)}{\pa(t,r)}=-\fr{t}{4F(t)U(t,r)V(t,r)},
$$
from which and \eqref{3.12} we find that $j$ does not vanish for $t>0$ which implies the mapping $(t,r)\mapsto(x,y)$ is a local one-to-one mapping. We claim that this mapping is globally one-to-one, including the line $t=0$. To show this, we only need to check the strict monotonicity of $\theta$ along curve $(1-\varpi)=\eps\geq0$. Using \eqref{4.2}, we compute
\begin{align}\label{4.7a}
 & \ \ \  (\theta_x, \theta_y)\cdot(\varpi_y, -\varpi_x)\\ &=t\theta_x\omega_y-t\theta_y\omega_x \nonumber   \\ &=\fr{\kappa+1-t^2}{t}[(F_1U+F_2V)(F_3U-F_4V)-(F_3U+F_4V)(F_1U-F_2V)] \nonumber  \\
&=\fr{2UV(\kappa+1-t^2)}{t}[F_2F_3-F_1F_4] \nonumber  \\
&=\fr{2UV(\kappa+1-t^2)}{t}\cdot2t\sqrt{1-t^2}=4\sqrt{1-t^2}(\kappa+1-t^2)UV<0.
\end{align}
The above inequality holds by the facts $U>0$ and $V<0$ by Lemma \ref{lem2}. Therefore, $\theta$ is a strictly decreasing function along each level curve of $(1-\varpi)\geq0$. Hence we have established the global one-to-one property of the mapping $(t,r)\mapsto(x,y)$.

\subsection{Solutions of system \eqref{2.10}}

We now construct the smooth supersonic solution to system \eqref{2.10}. Thanks to the mapping $(t,r)\mapsto(x,y)$ is globally one-to-one, then we can obtain the functions $t=t(x,y)$ and $r=r(x,y)$
and define the functions $(\theta,\varpi)$ by \eqref{3.1}
\begin{align}\label{4.8}
\theta=\hat{\theta}_1-r(x,y),\quad \varpi=\sqrt{1-t^2(x,y)},\ \ (x,y)\in PED,
\end{align}
where the region $PED$ is bounded with the curves $\widehat{PE}$, $\widehat{PD}$ and $\widehat{DE}$. The curves $\widehat{PD}$ and $\widehat{ED}$ are defined as follows
\begin{align}\label{4.9}
\begin{array}{l}
\widehat{PD}=\{(x,y)|\ \varpi(x,y)=1,\ x\in[x_1,x^*]\},\\
\widehat{DE}=\{(x,y)|\ r(x,y)=\check{r}(t(x,y)),\ x\in[x^*,x_2]\},
\end{array}
\end{align}
where the function $\check{r}(t)$ is given in \eqref{3.11} and the number $x^*=x(0,r^*)\in(x_1,x_2)$ is determined by \eqref{4.7}. We can also get the coordinates of point $D$ as $(x^*, y^*)$ with $y^*=y(0,r^*)$.
It is obvious that the functions $(\theta(x,y),\varpi(x,y))$ defined in \eqref{4.8} satisfy the boundary condition \eqref{2.17} by the construction of $(x(t,r), y(t,r))$ in \eqref{4.7}.

We next discuss the regularity of $(\theta(x,y),\varpi(x,y))$. Recalling \eqref{4.2} and applying \eqref{3.7} give
\begin{align}\label{4.10}
\begin{array}{l}
\dps\theta_x=\fr{F_1(t,r)\overline{V}-F_2(t,r)\overline{U}}{\overline{U}\cdot\overline{V}}, \quad
\theta_y=\fr{F_4(t,r)\overline{U}-F_3(t,r)\overline{V}}{\overline{U}\cdot\overline{V}}, \\[10pt]
\dps\varpi_x=-\fr{k+1-t^2}{\overline{U}\cdot\overline{V}}\bigg( \sin(\hat{\theta}_1-r)(\overline{U}+\overline{V}) +\sqrt{1-t^2}\cos(\hat{\theta}_1-r)\overline{W} \bigg),\\[10pt]
\dps\varpi_y= \fr{k+1-t^2}{\overline{U}\cdot\overline{V}}\bigg(\cos(\hat{\theta}_1-r)(\overline{U}+\overline{V}) -\sqrt{1-t^2}\sin(\hat{\theta}_1-r)\overline{W} \bigg),
\end{array}
\end{align}
from which and Lemmas \ref{lem2} and \ref{lem4}, we acquire that $(\theta_x, \theta_y, \varpi_x, \varpi_y)$ are uniformly bounded,  implying $\theta(x,y)$ and $\varpi(x,y)$ are uniformly Lipschitz continuous. We assert that the functions $(\theta_x, \theta_y, \varpi_x, \varpi_y)$ are actually uniformly $C^{\fr{1}{6}}$-continuous. To prove this assertion, we first establish a lemma.
\begin{lemma}\label{lem6}
Let $f(t,r)$ be a $C^{\fr{1}{3}}$ function defined on the whole region $P'E'D'$. Denote $\tilde{f}(x,y)=f(t(x,y), r(x,y))$. Then the function $\tilde{f}(x,y)$ is uniformly $C^{\fr{1}{6}}$-continuous
in the whole region $PED$.
\end{lemma}
\begin{proof}
Let $(x',y')$ and $(x'',y'')$ be any two points in $PED$, and let $(t',r')$ and $(t'',r'')$ be two points in $P'E'D'$ such that $t'=t(x',y'), r'=r(x',y')$ and $t''=t(x'',y''), r''=r(x'',y'')$. We calculate
\begin{align}\label{4.11}
&|\tilde{f}(x'',y'')-\tilde{f}(x',y')|=|f(t'',r'')-f(t',r')|\nonumber \\
\leq &M |(t'',r'')-(t',r')|^{\fr{1}{3}} = M\bigg( |t''-t'|^2 +|r''-r'|^2\bigg)^{\fr{1}{6}}
\end{align}
for some uniformly constant $M>0$. For the term $|r''-r'|$, one obtains
\begin{align}\label{4.12}
|r''-r'|=|\theta(x'',y'')-\theta(x',y')|\leq M|(x'',y'')-(x',y')|,
\end{align}
by the uniform Lipschitz continuity of $\theta(x,y)$. For the term $|t''-t'|^2$ we have
\begin{align}\label{4.13}
|t''-t'|^2\leq&|t''-t'|\cdot|t''+t'| =|{t''}^2-{t'}^2| \nonumber \\
=&|(1-\varpi^2(x'',y'')) -(1-\varpi^2(x',y'))| \nonumber \\
=& |\varpi^2(x'',y'') -\varpi^2(x',y')| \leq 2|\varpi(x'',y'') -\varpi(x',y')| \nonumber \\
\leq& M|(x'',y'')-(x',y')|,
\end{align}
in view of  the uniform Lipschitz continuity of $\varpi(x,y)$. Putting \eqref{4.12} and \eqref{4.13} into \eqref{4.11} leads to
\begin{align}\label{4.14}
|\tilde{f}(x'',y'')-\tilde{f}(x',y')|\leq M|(x'',y'')-(x',y')|^{\fr{1}{6}}
\end{align}
for some suitable uniformly constant $M>0$, which ends the proof of the lemma.
\end{proof}
In view of Lemmas \ref{lem5} and \ref{lem6}, the above assertion is valid and then the functions $(\theta(x,y), \varpi(x,y))$ are uniformly $C^{1,\fr{1}{6}}$-continuous up to the sonic curve $\widehat{PD}$.

Moreover, we apply \eqref{4.10} again to find that
$$
(\varpi_x)^2+(\varpi_y)^2 =\fr{(\kappa+1-t^2)^2}{\overline{U}^2\overline{V}^2}[(\overline{U}+\overline{V})^2+(1-t^2)\overline{W}^2],
$$
which,  together with Lemmas \ref{lem2} and \ref{lem4}, gives,
$$
0<\widetilde{m}\leq(\varpi_x)^2+(\varpi_y)^2\leq \widetilde{M}<\infty,
$$
for some constants $\widetilde{m}$ and $\widetilde{M}$, meaning that the curve $\varpi(x,y)=\eps\geq0$ is $C^{1}$-continuous. Furthermore, due to Lemmas \ref{lem5} and \ref{lem6}, the level curve $\varpi(x,y)$ and then the sonic curve $\widehat{PD}$ are actually $C^{1,\fr{1}{6}}$-continuous. The monotonicity of $\theta$ along $\widehat{PD}$ follows immediately from \eqref{4.7a}.

We now show that the curve $\widehat{DE}$ defined in \eqref{4.9} is a negative characteristic. Since $\widehat{D'E'}$ is a positive characteristic curve, then it suffices to show that the mapping $(t,r)\mapsto(x,y)$ defined in \eqref{4.7} transforms a positive characteristic curve in $t$-$r$ plane into a negative one in $x$-$y$ plane. To prove it, we differentiate the equality $r(x,y)=r_+(t(x,y))$ with respect to $x$ and use the fact $r_+'(t)=\lambda$ to get
\begin{align}\label{4.16}
\fr{{\rm d}y}{{\rm d}x}=-\fr{r_x-\lambda t_x}{r_y-\lambda t_y} =-\fr{\lambda\sqrt{1-t^2}\omega_x-\theta_x}{\lambda\sqrt{1-t^2}\omega_y-\theta_y}.
\end{align}
Inserting \eqref{4.2} into \eqref{4.16} and employing \eqref{4.8} arrives at
\begin{align*}
\fr{{\rm d}y}{{\rm d}x}&=-\fr{-\lambda\sqrt{1-t^2}\fr{\kappa+1-t^2}{t^2}(F_1U-F_2V)-(F_1U+F_2V)} {\lambda\sqrt{1-t^2}\fr{\kappa+1-t^2}{t^2}(F_3U-F_4V)-(-F_3U-F_4V)} \\
&=\fr{F_1}{F_3}=\fr{\cos\omega\sin\theta-\sin\omega\cos\theta}{\cos\omega\cos\theta+\sin\omega\sin\theta} =\fr{\sin\beta}{\cos\beta}=\Lambda_-.
\end{align*}
Thus the curves defined by the equality $r(x,y)=r_+(t(x,y))$ are negative negative characteristics in $x$-$y$ plane. It is not difficult to verify that $\theta$ is a monotone decreasing function along curve $\widehat{DE}$.

Finally, we check that the functions $\theta(x,y)$ and $\varpi(x,y)$ defined in \eqref{4.8} satisfy system \eqref{2.10}. By performing a direct calculation, we see by \eqref{4.10} that
\begin{align}\label{4.17}
\bar\pa^+\theta &= \cos\alpha \theta_x+\sin\alpha \theta_y \nonumber \\
&=\fr{1}{\overline{U}\cdot\overline{V}}\bigg\{\cos\alpha(F_1\overline{V}-F_2\overline{U}) +\sin\alpha(F_4\overline{U}-F_3\overline{V})\bigg\} =-\fr{2\cos\omega\varpi}{\overline{U}},
\end{align}
and
\begin{align}\label{4.18}
\bar\pa^+\varpi&=\cos\alpha \varpi_x+\sin\alpha \varpi_y \nonumber \\
&=\fr{k+1-t^2}{\overline{U}\cdot\overline{V}}\bigg\{ \sin\alpha[\cos(\hat{\theta}_1-r)(\overline{U}+\overline{V}) -\sqrt{1-t^2}\sin(\hat{\theta}_1-r)\overline{W}] \nonumber \\
&\qquad \qquad \qquad \quad   -\cos\alpha[\sin(\hat{\theta}_1-r)(\overline{U}+\overline{V}) +\sqrt{1-t^2}\cos(\hat{\theta}_1-r)\overline{W}]  \bigg\}  \nonumber \\
&=\fr{2(\kappa+\varpi^2)\varpi}{\overline{U}}.
\end{align}
Combining with \eqref{4.17} and \eqref{4.18} yields
\begin{align*}
\bar\pa^+\theta +\fr{\cos\omega}{\kappa+\varpi^2}\bar\pa^+\varpi =-\fr{2\cos\omega\varpi}{\overline{U}} +\fr{\cos\omega}{\kappa+\varpi^2}\cdot \fr{2(\kappa+\varpi^2)\varpi}{\overline{U}}=0,
\end{align*}
which means that the first equation of \eqref{2.10} holds. The second equation of \eqref{2.10}
can be checked analogously. The proof of Theorem \ref{thm1} is completed.

\section*{Acknowledgements}

The authors would like to thank the referee for very helpful comments and suggestions to improve the quality of the paper.


\medskip
Received xxxx 20xx; revised xxxx 20xx.
\medskip

\end{document}